\documentclass[11pt]{article}
\usepackage{amsmath}
\usepackage{amssymb,amsbsy}
\usepackage{graphicx}
\usepackage{booktabs}
\usepackage{dsfont}
\usepackage{natbib}
\usepackage[margin=1in]{geometry}
\usepackage{tikz}
\usetikzlibrary{arrows,decorations.pathmorphing,backgrounds,positioning,fit,petri,mindmap,shapes.geometric,decorations.pathreplacing}

\usepackage{pdfsync}

\newcommand{\Tcal}{\mathcal{T}}
\newcommand{\diam}{\operatorname{diam}}

\newtheorem{theorem}{Theorem}
\newtheorem{proposition}{Proposition}
\newtheorem{lemma}{Lemma}

\renewcommand{\phi}{\varphi}

\newcommand{\N}{\mathbb{N}}
\newcommand{\R}{\mathbb{R}}

\newcommand{\cT}{\mathcal{T}}

\def\ds1{\mathds{1}}
\renewcommand{\epsilon}{\varepsilon}

\newlength{\minipagewidth}
\newcommand{\beq}{\begin{equation}}
\newcommand{\eeq}{\end{equation}}

\newcommand{\beqa}{\begin{eqnarray}}
\newcommand{\eeqa}{\end{eqnarray}}

\newcommand{\beqan}{\begin{eqnarray*}}
\newcommand{\eeqan}{\end{eqnarray*}}

\def\ba#1\ea{\begin{align*}#1\end{align*}} %\ba = \begin{algin*}, \ea = \end{align*}
\def\banum#1\eanum{\begin{align}#1\end{align}} %\banum = \begin{algin}, \eanum

\newcommand{\BlackBox}{\rule{1.5ex}{1.5ex}}  % end of proof
\newenvironment{proof}{\par\noindent{\bf Proof\ }}{\hfill\BlackBox\\[2mm]}
\newenvironment{proof2}{\par\noindent{\bf Proof of the first part of theorem 2:\ }}{\hfill\BlackBox\\[2mm]}
\newenvironment{proof3}{\par\noindent{\bf Proof of the second part of theorem 2:\ }}{\hfill\BlackBox\\[2mm]}

\begin{document}
\title{On paths, stars and wyes in trees}
\author{S{\'e}bastien Bubeck \and Katherine Edwards \and Horia Mania \and Cathryn Supko}
\date{}

\maketitle

\abstract{We further the study of local profiles of trees. Bubeck and Linial showed that the set of $5$-profiles contains a certain polytope, namely the convex hull of $d$-millipedes, and they proved that the segment $[0\text{-millipede}, 1\text{-millipede}]$ corresponds to a face of the set of $5$-profiles. Our main result shows that the segment $[1\text{-millipede}, 2\text{-millipede}]$ also corresponds to a face. Surprisingly we also show that for $d \geq 4$ the segment $[d\text{-millipede}, (d+1)\text{-millipede}]$ is not a face of the set of $5$-profiles. We do so by exhibiting new trees which are generalized millipedes with intriguing patterns for their degree sequence. The plot thickens, and the set of $5$-profiles remains a mysterious convex set.}

\section{Introduction}
We study the notion of local profiles of trees introduced by Bubeck and Linial in \cite{BL13} (henceforth BL). We focus mostly on $5$-profiles: for a tree $T$ let $P(T)$ be the number of $5$-vertex paths in $T$, $S(T)$ the number of $5$-vertex stars in $T$, and $Y(T)$ the number of $5$-vertex ``wyes'' in $T$. We simply write $S, P$ and $Y$ when $T$ is clear from the context. BL proved that all trees satisfy the following linear relation between those quantities: $Y \leq 36 S + P + 4$. Our main contribution is to prove a tighter bound that was suggested in [Open problem 1, BL].
\begin{theorem} \label{th:Y9SP}
All trees satisfy
$$Y \leq 9 S + P + 6.$$
\end{theorem}
This bound is optimal in the sense that for any $M \in \N$ there exists a tree $T$ with $P, S, Y \geq M$ and $Y=9 S + P$ (see [Section 5, BL]). In particular as explained in [Open problem 1, BL] this new bound characterizes one of the face of the set of $5$-profiles $\Delta_{\Tcal}(5)$ (see below for a precise definition of the set of $k$-profiles $\Delta_{\Tcal}(k)$) which corresponds to the convex hull of $1$-millipede and $2$-millipede.
Yet, perhaps surprisingly, we exhibit in Section \ref{sec:examples} new trees whose $5$-profiles are outside the convex hull of simple millipides (thus proving that [$(9)$, BL] is a strict inclusion and completing the answer to [Open Problem 1, BL]). 
\newline

We also answer positively [Open problem 3, BL] (a similar result was recently and independently obtained in \cite{CSW15}): we prove that if the proportion of $k$-vertex paths among $k$-vertex subtrees (denoted $p_1$, see below for the more precise definition) goes to zero, then the proportion of $k$-vertex stars (denoted $p_2$) goes to $1$. In fact our explicit bound given below also gives a partial answer to [Open problem 7, BL] as it is a non-linear relation between $p_1$ and $p_2$.

\begin{theorem} \label{th:fewpaths}
\label{th: FPMS}

Let $k\geqslant 6$ and $p\in \Delta_{\Tcal}(k)$, then
\[
p_2 \geqslant 1 - e(k- 1)!(k-1)^\epsilon p_1^{1 - \frac{\epsilon}{k - 1}}\text{, where }\epsilon = \frac{k-2 + \sqrt{(k-2)^2 + 4(k-2)}}{2} .
\]
If $p \in \Delta_{\Tcal}(5)$, then 
\[
p_2 \geqslant 1 - p_1 - 2\cdot 4^{2+\sqrt{3}}p_1^{\frac{2-\sqrt{3}}{4}}.
\]
\end{theorem} 

For sake of convenience for the reader we recall the precise definition of $k$-profiles. For (unlabelled) trees $T$, $R$, we denote by $c(R, T)$ the number of copies of $R$ in $T$, or in other words the number of injective homomorphism from $R$ to $T$. Let $T_1^k, \hdots, T_{N_k}^k$ be a list of all isomorphism types of $k$-vertex trees. The $k$-profile of a tree $T$ is the vector $p^{(k)}(T) \in \R^{N_k}$ whose $i$-th coordinate is
$$(p^{(k)}(T))_i = \frac{c(T^k_i, T)}{Z_k(T)}, \ \text{where} \ Z_k(T) = \sum_{j=1}^{N_k} c(T^k_j, T).$$
We focus on the set of $k$-profiles attainable with large trees:
$$\Delta_{\cT}(k) = \left\{p \in \R^{N_k} : \exists (T_n), |T_n| \xrightarrow[n \rightarrow \infty]{} \infty, \, \text{and} \, p^{(k)}(T_n) \xrightarrow[n \rightarrow \infty]{} p \right\} ,$$
where $|T|$ denotes the number of vertices in $T$. In the rest of the paper $D(T)$ denotes the largest degree in $T$, and $d_T(v)$ denotes the degree of vertex $v$ in $T$ (when the tree is clear from the context we drop the reference to $T$).

\section{Proof of Theorem \ref{th:Y9SP}} \label{sec:Y9SP}
Our approach is quite different from BL's proof of the weaker bound $Y \leq 36 S + P + 4$. In particular we use simple inductive arguments which avoid the pedestrian counting that BL used to rewrite $P - Y$ as a linear combination of the number of vertices of certain ``types''. On the other hand the inductive step in the BL proof relied on a special way of cutting a tree, and our proof is centered around an extension of such cuts. Precisely for $i,j \in \N$ we define the $(i,j)$-cut of a tree $T$ around $(u, v)$ (where $\{u,v\}$ is an edge) as follows:
remove the edge $\{u,v\}$, and add a path of length $i$ to $u$ and a path of length $j$ to $v$. When $u,v, i$, and $j$ are clear from the context, we denote $T_1$ and $T_2$ the two trees in the forest obtained after the $(i,j)$-cut, such that $u \in T_1$ and $v \in T_2$. The BL proof used only $(1,0)$-cuts, while we use $(0,0)$, $(1,0)$, $(1,1)$ and $(2,1)$-cuts.

The proof of Theorem \ref{th:Y9SP} proceeds in two steps: first we prove it for trees such that $D \leq 4$, and then we extend it to the general case. This is similar to what BL did, as they first proved $Y \leq 36 S + P + 4$ when $D \leq 3$ (in which case it rewrites $Y \leq P + 4$) and then extended to the general case.

We start with a simple lemma which shows that, without loss of generality, we can focus on trees without degree $2$ vertices.

\begin{lemma}\label{lem:deg2}
If there exists a tree such $Y > 9 S + P + 6$, then there exists a tree without degree $2$ vertices which also satisfy this inequality.
\end{lemma}

\begin{proof}
Assuming that there exists trees with $Y > 9 S + P + 6$, let $T$ be the smallest such tree. We will now show by contradiction that $T$ cannot have degree $2$ vertices. Let us assume that it does. Then there exist vertices $v$ and $w$, each with degree at least $3$ that are joined by a path $P$ of length at least $2$, all of whose internal vertices have degree $2$.
Let the tree $T'$ be obtained from $T$ by replacing $P$ with a single edge $vw$.
Clearly $S(T') = S(T)$.
We also have $$Y(T') = Y(T) + \binom{d(v)-1}{2}(d(w) - 2) + \binom{d(w) - 1}{2}(d(v) - 2)$$
and $$P(T') \leq P(T) + \sum_{x\sim v \atop x\notin V(P)}(d(x)-1)(d(w) - 2) + \sum_{x\sim w \atop x\notin V(P)}(d(x) - 1)(d(v) - 2).$$
By the minimality of $T$ we have $Y(T') \leq 9S(T') + P(T') + 6$, and so $P(T') - P(T) > Y(T') - Y(T)$.
In particular we may assume that
\begin{equation}\label{eq:balance}
\sum_{x\sim w \atop x\notin V(P)}(d(x)-1) > \binom{d(w) - 1}{2}.
\end{equation}
We now consider two cases:

\noindent\textbf{Case 1:} $d(v)\geq 4$ or $d(w)\geq 4$.\\
We consider the $(2,1)$-cut around $(v,w)$. Observe first that $S(T) = S(T_1) + S(T_2)$.
Now 
$$Y(T) - Y(T_1) - Y(T_2) = \binom{d(w)-1}{2}.$$
Further 
$$P(T) - P(T_1) - P(T_2) \geq (d(v)-1) + (d(w) - 1) +  \sum_{x\sim w \atop x\notin V(P)}(d(x) - 1).$$
Therefore 
\begin{align*}
Y(T) &\leq Y(T_1) + Y(T_2) + \binom{d(w)-1}{2}\\
 &\leq 9S(T) + P(T) + 12 + \binom{d(w)-1}{2} - (d(v)-1) - (d(w) - 1) - \sum_{x\sim w\atop x\notin V(P} (d(x) - 1).
\end{align*}
The lemma follows from (\ref{eq:balance}).

\noindent\textbf{Case 2:} $d(v) = d(w) = 3$.\\
We consider the $(1,1)$-cut around $(v,w)$.
Again $S(T) = S(T_1) + S(T_2)$.
Now we have  
$$Y(T) - Y(T_1) - Y(T_2) = 2$$
and 
\begin{align*}
& P(T) - P(T_1) - P(T_2) \\
& \geq (d(v)-1) + (d(w) - 1) + \sum_{x\sim v \atop x\notin V(P)}(d(x) - 1) +  \sum_{x\sim w \atop x\notin V(P)}(d(x) - 1) \\
& \geq 6 + \sum_{x\sim v \atop x\notin V(P)}(d(x) - 1) ,
\end{align*}
where the last inequality used (\ref{eq:balance}). Now observe that if 
$$\sum_{x\sim v \atop x\notin V(P)}(d(x) - 1) \geq 2, \ \text{or} \ \sum_{x\sim v \atop x\notin V(P)}(d(x) - 1) = 0 ,$$
then we are done. On the other other hand if this sum is equal to $1$, then $v$ is adjacent to two vertices of degree $2$ and one leaf $x$.
But then consider $T'' = T\setminus x$.
We have $Y(T'') = Y(T) - 2$ and $P(T'') \leq P(T) - 2$, contradicting the minimality of $T$.
This completes the proof of the lemma.
\end{proof}

\begin{proposition}
If $D \leq 4$, then $Y \leq 9 S + P + 6$.
\end{proposition}

\begin{proof}
We say that a $5$-vertex subtree $S$ of $T$ is centered at $v\in V(T)$ if $v\in V(S)$ and $S$ is a star and $v$ has degree $4$ in $S$; if $S$ is a path and $v$ is the middle vertex; or if $S$ is a wye and $v$ has degree $3$ in $S$.
For each vertex $v\in V(T)$ denote by $S(v), P(v) $ and $Y(v)$ the number stars, paths and wyes, resp., centred at $v$.
Let us also define $\gamma(v) = 2 - d'(v)$, where $d'(v)$ denotes the number of non-leaf neighbours of $v$. Observe that $ \sum_{v \ \text{non-leaf}}\gamma(v) = 2$. Thus is enough to show that $9 S(v) + P(v) - Y(v) + 3 \gamma(v) \geq 0$ for any non leaf-vertex $v$. Recall also the formulas
$$S(v) = \binom{d(v)}{4}$$
and 
$$P(v) = \sum_{u,w \sim v}(d(u) - 1)(d(w) - 1) $$
and 
$$Y(v) = \sum_{u \sim v}(d(u) - 1)\binom{d(v) - 1}{3}. $$
We now consider two cases:

\noindent\textbf{Case 1:} $d(v) = 4$.\\
It suffices to verify that for each $k\leq 4$ and every choice of $x_1,\dots,x_k \in \{2,3\}$ we have
$$ \sum_{i=1}^{k}3(x_i + 1) - 1 \leq 15 + \sum_{1\leq i < j \leq k}x_ix_j ,$$
which is indeed true.

\noindent\textbf{Case 2:} $d(v) = 3$.\\
It suffices to verify that for each $k\leq 3$ and every choice of $x_1,\dots,x_k \in \{2,3\}$ we have
$$ \sum_{i=1}^{k}(x_i + 3) -1 \leq 6 + \sum_{1\leq i < j \leq k}x_ix_j.$$
which is again true.
\end{proof}

We can now move to the proof of Theorem \ref{th:Y9SP}.

\begin{proof} We prove the inequality by induction. Let $u$ be a vertex of maximal degree in $T$, such that at most one ot its neighbors has degree $D(T)$ (clearly such a vertex exist). Let $v$ be a neighbor of $u$ of maximal degree. Denote $a=d(v)$ and $d=D(T)$. We will either consider the $(0,0)$-cut or the $(0,1)$-cut of $T$ around $(u,v)$, and then apply the induction hypothesis on $T_1$ and $T_2$, i.e., $Y(T_i) \leq 9 S(T_i) + P(T_i) + 6$. Thus we have to show that the number of paths removed (denoted $\delta_P$), plus nine times the number of stars removed (denoted $\delta_S$), is at least six plus the number of wyes removed (denoted $\delta_Y)$.

In the case of a $(0,0)$-cut one has:
\begin{align*}
\delta_S &= 9{d\choose 4} - 9{d-1 \choose 4} + 9{a\choose 4} - 9{a-1 \choose 4} \\
&= \frac{3}{2}(d-1)(d-2)(d-3) + \frac{3}{2}(a-1)(a-2)(a-3) , 
\end{align*}
and
\begin{align*}
\delta_P & \geq (a-1) \sum_{\substack{w\sim u\\w \neq v}} (d(w) - 1) + (d-1) \sum_{\substack{w \sim v\\w \neq u}}(d(w) - 1) ,
\end{align*}
and
\begin{align*}
\delta_Y &= {a- 1\choose 2}(d-1) + {d - 1\choose 2}(a - 1) +\\ &+ \sum_{\substack{w\sim u\\ w\neq v}}{d(w) - 1 \choose 2} + (d-2)\left(\sum_{\substack{w \sim u \\ w \neq v}} d(w) - 1\right) +
\sum_{\substack{w\sim v\\ w\neq u}}{d(w) - 1 \choose 2} + (a-2)\left(\sum_{\substack{w \sim v \\ w \neq u}} d(w) - 1\right) 
\end{align*}
while for the $(0,1)$-cut:
\begin{align*}
\delta_S &= \frac{3}{2}(d-1)(d-2)(d-3) ,
\end{align*}
and
\begin{align*}
\delta_Y &= {a- 1\choose 2}(d-1) + {d - 1\choose 2}(a - 1) + \sum_{\substack{w\sim u\\ w\neq v}}{d(w) - 1 \choose 2} + (d-2) \sum_{\substack{w \sim u \\ w \neq v}} (d(w) - 1) ,
\end{align*}
and the same lower bound (as for the $(0,0)$-cut) on $\delta_P$ holds. 

We now consider three cases. Recall that by the choice of $u$ and $v$ we know that for $w \sim u$,$d(w) \leqslant \min{\{a, d - 1\}}$ and for $w \sim v$, $d(w) \leqslant d$. 

\noindent 
\textbf{Case 1: $a=d$.} \\
We use the $(0,0)$-cut. It is enough to prove that
\begin{align*}
\label{to check}
3(d-1)(d-2)(d-3) &\geqslant 2{d- 1\choose 2}(d-1) +\\&+ \frac{1}{2}(d-1)^2(d-4) + \frac{1}{2}(d-1)(d-2)(d-5) + 6.
\end{align*}
This is equivalent to checking $d^2 - 6d + 9 \geqslant \frac{6}{d-1}$, which holds for $d\geq 5$.

\noindent 
\textbf{Case 2: $a = d- 1$.}\\
We use the $(0,0)$-cut. It is enough to prove that
\begin{align*}
&\frac{3}{2}(d-1)(d-2)(d-3) + \frac{3}{2}(d-2)(d-3)(d-4) \\ &\geq {d-2\choose 2}(d-1) + {d - 1\choose 2}(d-2) + (d-1){d-2\choose 2} +\frac{1}{2}(d-2)(d-1)(d-6) + 6. 
\end{align*}
This is equivalent to checking that $2d^2 - 15d + 31 \geqslant \frac{12}{d-2}$, which holds for $d\geq 5$. 

\noindent 
\textbf{Case 3: $a\leq d- 2$.}\\
We use the $(0,1)$-cut. It is enough to prove that
\begin{align*}
& \frac{3}{2}(d-1)(d-2)(d-3) + (a-1) \sum_{\substack{w \sim u\\ w \neq v}}(d(w) - 1) \\
&\geq  {a- 1\choose 2}(d-1) + {d - 1\choose 2}(a - 1) +
 \sum_{\substack{w \sim u\\ w \neq v}}{d(w) - 1 \choose 2} + (d-2) \sum_{\substack{w \sim u \\ w \neq v}} (d(w) - 1) + 6
\end{align*}

Recall that for any $w \sim u$, $d(w) \leqslant a$, then it is enough to check
\begin{align*}
& \frac{3}{2}(d-1)(d-2)(d-3) \\
& \geq  {a- 1\choose 2}(d-1) + {d - 1\choose 2}(a - 1) + (d-1){a - 1\choose 2} + (d - a - 1)(d-1)(a-1) + 6 \\
& = {d - 1\choose 2}(a - 1) + (d-3)(d-1)(a-1) + 6.
\end{align*}
The last term in the above display is maximized for $a = d-2$, and thus it is enough to check
\[
\frac{3}{2}(d-1)(d-2)(d-3) \geq  (d-3)(d-4)(d-1) + {d - 1\choose 2}(d - 3) + (d-1)(d-3) + 6.
\]
This inequality is equivalent to $1\geq \frac{6}{(d-1)(d-3)}$ which holds for $d\geq 5$.
\end{proof}

\section{Some profiles outside the convex hull of simple millipedes} \label{sec:examples}

\begin{figure}[ht]
	\centering
	\scalebox{0.5}{\input{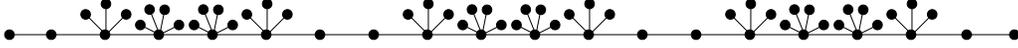}}
	\caption{A $(0,0,3,4,4,3)$-millipede.}
	\label{fig:milli}
\end{figure}

In this section we give a negative answer to the question posed in BL's open problem 1.
Recall that we denote by $\mathcal P(k)$ the projection of $\Delta(k)$ onto the first two coordinates (those corresponding to paths and stars on $k$ vertices).

Let $D = (d_1,\dots, d_{\ell})$ be a finite sequence of integers. 
We say a graph $T$ is a $D$-millipede if $T$ consists of a path on vertices $v_1,\dots,v_n$, along with $d_{i} + 2$ pendant vertices at each vertex $v_j$ with $j = i (\textrm{mod } \ell)$ (see Figure \ref{fig:milli}). We say that the $D$-millipede has length $n$.
For a fixed sequence $D$, we write $T_n^D$ to denote the $D$-millipede of length $n$.

\begin{figure}[t]
\begin{center}
\includegraphics[scale=0.8]{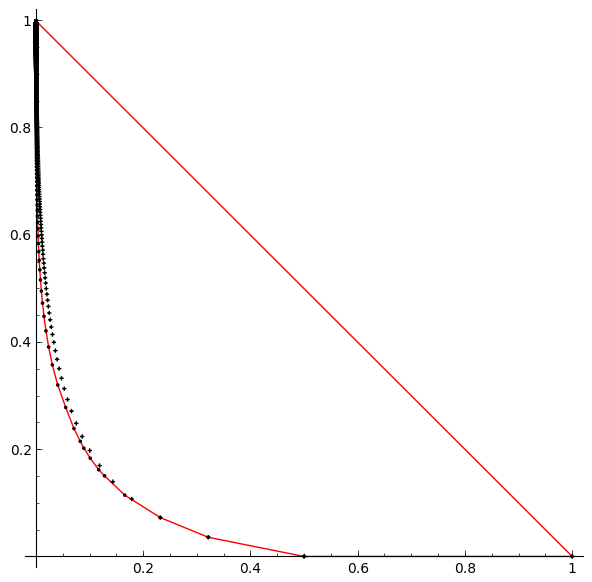}
\end{center}\caption{Points marked with `+' denote limiting profiles of $(d)$-millipedes. The red curve bounds the convex hull of a larger set of profiles.}\label{fig:counterexamples}
\end{figure}

BL asked whether all points in $\mathcal P(5)$ lie inside the convex hull of profiles of the sequences $(T_n^{(d)})$ of $(d)$-millipedes for $i\geq 0$.
In fact $\mathcal P(5)$ contains the convex hull of the limiting $5$-profiles corresponding to the following sequences of millipedes:
\begin{itemize}
\item $(d)$-millipedes; $0\leq d \leq 3$
\item $(0,0,3,4,4,3)$-millipede
\item $(0,0,d,d+2,d+2,d)$-millipedes; $3\leq d \leq 5$
\item $(0,0,d,d+1,d)$-millipedes; $4\leq d \leq 6$
\item $(0,0,d,d)$-millipedes; $d\geq 6$
\end{itemize}

This set strictly contains the limiting $5$-profiles of $(d)$-millipedes.
It is straightforward to compute the limiting profiles of the sequences $(T_n^D)$ for each of the sequences describes above. 
We omit these tedious calculations and instead direct the reader's attention to Figure \ref{fig:counterexamples} for comparison.

In order to prove that the red curve in Figure \ref{fig:counterexamples} is tight, it would be sufficient to prove that each
of an (infinite) sequence of linear inequalities holds for every tree. 
The first inequality is our Theorem \ref{th:Y9SP}, that $Y\leq 9S + P + 6$.
The next few inequalities would be $Y \leq \tfrac{144}{29} S +  \tfrac{42}{29} P + k_1$, then $Y \leq \tfrac{624}{175} S +  \tfrac{66}{35} P + k_2$ and $Y\leq \tfrac{107}{40} S +  \tfrac{5}{2} P + k_3$ (for universal constants $k_i$).

%%%%%%%%%%%%%%%%%%%%%%%%%%%%%%%%%%%%%%%%%%%%%%%%%%%%%%%%%%%%%%%%%%%%%%%%%%%%%%%%%%%%
%%%%%%%%%%%%%%%%%%%%%%%%%%%%%%%%%%%%%%%%%%%%%%%%%%%%%%%%%%%%%%%%%%%%%%%%%%%%%%%%%%%%
%%%%%%%%%%%%%%%%%%%%%%%%%%%%%%%%%%%%%%%%%%%%%%%%%%%%%%%%%%%%%%%%%%%%%%%%%%%%%%%%%%%%
%%%%%%%%%%%%%%%%%%%%%%%%%%%%%%%%%%%%%%%%%%%%%%%%%%%%%%%%%%%%%%%%%%%%%%%%%%%%%%%%%%%%

\section{Few paths implies many stars} \label{sec:fewpaths}

This section is dedicated to the proof of Theorem \ref{th: FPMS}. This result offers a lower-bound on how fast the proportion of stars in a sequence of trees goes to one as the proportion of paths in the sequence goes to zero.

The key idea in the proof of Theorem \ref{th: FPMS} is to show that the number of non-star subtrees in a tree $T$ is $\mathcal{O}\left(\sum d_i^\epsilon \right)$ for some $\epsilon < k-1$, where $d_1, d_2,..., d_{|T|}$ is the degree sequence of $T$. If we can prove such a bound on the number of non-star $k$ subtrees $R_k(T)$, the theorem follows because $S = \Omega\left(\sum d_i^{k-1} \right)$. Why should one expect that there exists such a bound on $R_k(T)$? There are at most $(k-1)!D(T)^{k-1}$ subtrees with $k$ nodes that contain a fixed vertex of maximum degree. Then an inductive argument yields an upper bound with $\epsilon  = k-1$, but this bound is to weak for our purposes. If all the non-leaf nodes of $T$ have the same degree, by the same method we obtain a bound with $\epsilon = k -2$ by removing one at a time nodes $v$ with $d_v - 1$ neighboring leafs. The next proposition finds a middle ground. 

\begin{proposition}
\label{proposition: upperbound non-stars}
Let $T$ be a tree with degree sequence $d_1$, $d_2$,..., $d_n$. Then
\[
R_k(T) \leqslant e(k-1)!\sum_{d_i \geqslant 2} d_i^{\epsilon}, \text{ where } \epsilon = \frac{k-2 + \sqrt{(k-2)^2 + 4(k-2)}}{2} .
\]
\end{proposition}
\begin{proof}

Label $T$'s nodes by $v_1$, $v_2$,..., $v_n$ such that the first $m$ are all the nodes with degree at least $2$. Of course, $d_i$ denotes the degree of node $v_i$. 
We employ the following strategy. We build the tree $T$ from the empty set by adding some nodes at each step. Then we upper-bound the number of non-star trees added to the tree at each step and use these intermediate bounds to obtain the desired bound on $R_k(T)$. 

The following construction of $T$ meets our needs. Start with the tree $T_1$: a star with $d_1 + 1$ nodes, centered at $v_1$. Then, at each step, to construct $T_{r+1}$ from $T_r$ we choose a leaf node, which we label $v_{r+1}$, and attach to it $d_{r+1} - 1$ nodes in order to transform $v_{r+1}$ into a node of degree $d_{r+1}$.  Through appropriate choices of leafs at each step and maybe a relabeling of the nodes we can construct $T_m = T$. 

Denote by $S_r$ the set of non-star $k$-subtrees of $T_r$ that contain at least one of the leafs added at step $r$. Hence, $R_k(T) = \sum_{r = 2}^m |S_r|$ because $S_1 = \emptyset$. To upper-bound $|S_r|$ for all $r\geq 2$ we estimate the number of subtrees $S\in S_r$ based on the degrees of their nodes when viewed as nodes of $T$. 

We denote by $U(S)$ the maximum degree of a non-leaf node of the $k$-subtree $S$ when viewed as a node of $T$. We refer to the quantity $U(S)$ as the underlying degree of $S$ in $T$. Now, for some $\alpha >0$ to be chosen later, the sets $S_r$ can be written as a disjoint union: \[S_r = \{S\in S_r\colon U(S) \leqslant d_r^\alpha\}\bigsqcup \left(\bigsqcup_{u > d_r^\alpha} \{S\in S_r\colon U(S) = u\}\right).\] 

First, we bound the number of subtrees in $S_r$ with underlying degrees at most $d_r^\alpha$. Since each such subtree is not a star, it contains the edge connecting $v_r$ to $T_{r-1}$. Fix this edge and call it $e^*$. We upper-bound the number of $k$-subtrees with underlying degree at most $d_r^\alpha$ that can be obtained by starting with $e^*$ and adding new edges one at a time. Therefore, $k-2$ edges need to be added in order to obtain a $k$-subtree. If any of the endpoints of $e^*$ has degree in $T$ larger than $d_r^\alpha$, then there are no $S\in S_r$ with $U(S) \leq d_r^\alpha$. On the other hand, if this this is not the case, there are at most $d_r^\alpha$ choices for a new edge around each of the endpoints of $e^*$. So there are at most $2d_r^\alpha$ choices for the second edge of a subtree $S$. Once a choice is made so that the underlying degree remains at most $d_r^\alpha$, there can be at most $d_r^\alpha$ new choices that can appear from adding a new edge. This means that after adding $l$ edges, there are at most $(l+2)d_r^\alpha$ possible choices for a new edge. Hence, there are at most $\prod_{l = 0}^{k-3} d_r^\alpha(l+2) = (k-1)!d_r^{\alpha(k-2)}$ subtrees in $S_r$ with underlying degree at most $d_r^\alpha.$ 

Now, each tree $S\in S_r$ with $U(S) > d_r^\alpha$ contains a node $v_i$, a non-leaf in $S$, with degree in $T$ equal to $U(S)$, meaning $d_i = U(S)$. Observe that all the nodes on the path connecting $v_i$ and $v_{r}$ have degrees at most $U(S) = d_i$ in $T$ by the definition of the underlying degree. Denote the distance between $v_i$ and $v_r$ with $l_i$, thus $1 \leqslant l_i \leqslant k-3$. By the same argument as in the one in the previous paragraph, there are at most $\frac{(k - 1)!}{l_i!}d_rd_i^{k - 2 - l_i}\leqslant \frac{(k-1)!}{l_i!}d_i^{k-2-l_i + 1/\alpha}$ such subtrees that contain both $v_r$ and $v_i$. Then
\[
\left|S_r\right| = \left|\bigsqcup_{u\geqslant 2}\{S\in S_r\colon U(S) = u\}\right| \leqslant (k-1)!d_r^{\alpha(k-2)} + (k-1)!\sum_{v_i}\frac{1}{l_i!}d_i^{k - 2 - l_i + \frac{1}{\alpha}},
\]
where the sum is taken over all nodes $v_i$ in $T_{r-1}$ at distance at most $k-3$ from $v_r$ with $d_i > d_r^{\alpha}$, and such that all the nodes on the path connecting $v_i$ to $v_r$ have degrees at most $d_i$. Summing over $r$ gives an upper bound on $R_k(T)$.

For each $i\in [m]$ we upper bound the sum of all the terms containing a power of $d_i$ that appear in the summation over $r$. It is easy to see that this sum is upper bounded by
\[
(k-1)!d_i^{\alpha(k-2)} + (k-1)!\sum_{l = 1}^{k-3}\frac{1}{l!}d_i^ld_i^{k- 2 -l + \frac{1}{\alpha}} \leqslant (k-1)!d_i^{\alpha(k-2)} + (e - 1)(k-1)!d_i^{k - 2 + \frac{1}{\alpha}}
\]
because in $T$ there are at most $d_i^l$ nodes at depth $l$ away from $v_i$ such that all the nodes on the path connecting them to $v_i$ have degrees at most $d_i$ (here $e$ denotes Euler's constant). To obtain the conclusion we set $\alpha(k-2) = k - 2 + \frac{1}{\alpha}$. 
\end{proof}

Remark that we found a bound on $R_k(T)$ with $k-2 < \epsilon < k-1$. This bound is the main ingredient in the proof of the first part of Theorem \ref{th: FPMS}.

Before we go into the proof of the theorem we reiterate the idea of gluing trees, used in [BL] to show that the set $\Delta_\mathcal{T}(k)$ is convex. 
Given two trees $S$ and $T$ we construct a new tree $S\boxtimes_k T$ by choosing a leaf in $S$ and a leaf in $T$, and unite them with a path that contains $k-1$ newly added nodes. Hence the distance between the two chosen leafs in $S\boxtimes_k T$ is $k$. This construction depends on the choice of leafs, but this will not be an impediment for our future work. 

\begin{lemma}
\label{lemma: gluing inequalities}
Let $S$ and $T$ be two trees. Then for all $1\leqslant i \leqslant N_k$
\[
c(T_i^k, S) + c(T_i^k, T) \leqslant c(T_i^k, S\boxtimes_k T) \leqslant c(T_i^k, S) + c(T_i^k, T) + (k-2)!D(S)^{k-3} + (k-2)!D(T)^{k-3} 
\]
\[
Z_k(S) + Z_k(T) \leqslant Z(S\boxtimes_k T) \leqslant Z(T_i^k, S) + Z(T_i^k, T) + (k-2)!D(S)^{k-3} + (k-2)!D(T)^{k-3}. 
\]
\end{lemma}
\begin{proof}
The left hand side inequalities are trivially true. We prove the upperbound on $Z_k(S\boxtimes_k T)$, the argument is based on similar ideas to the ones used in the proof of Proposition \ref{proposition: upperbound non-stars} and it will imply the other inequalities also. 

We need to upperbound the number of $k$-subtrees that are formed by attaching a new vertex to a leaf node of $S$. It is clear that such a tree contains the new vertex, the leaf and the parent of the leaf. Hence, we need to count in how many ways $k-3$ nodes can be chosen from $S$ to form a $k$-subtree. We do this iteritatively. At the first step there are at most $D(S)$ choices and after choosing each node, there can be at most $D(S)$ new choices. Then, at the second step there can be at most $2D(S)$ choices, at the third step $3D(S)$ and so forth. This means that there are at most $(k-3)!D(S)^{k-3}$ subtrees with $k$ nodes that contain the new vertex. By the same procedure we see that there are at most 
\[
\sum_{i = 0}^{k-3} (k-3 - i)!D(S)^{k - 3 - i} + 1 \leqslant (k-2)!D(S)^{k-3}
\]  
$k$-subtrees formed by the $k-1$ vertices added by the gluing procedure together with $S$'s nodes. A similar bound holds for $T$ and the upper bound for $Z_k(S\boxtimes_k T)$ follows. Remark that in this argument we upper-bounded the number of all $k$-subtrees that contain the new nodes, thus the other inequalities follow as well. 
\end{proof}

The next lemma will offer us a way to lower-bound the number of paths of the trees in a convergent sequence.

\begin{lemma}
\label{lemma: new sequence through gluing}
Let $(T_n)_{n\geqslant 1}$ be a sequence of trees such that $p^{(k)}(T_n) \rightarrow p$ and such that $D(T_n)$ is non-decreasing, and let \[T'_n = T_1\boxtimes_k T_2 \boxtimes_k ... \boxtimes_k T_n.\]
Then $p^{(k)}(T'_n) \rightarrow p$.
\end{lemma}
\begin{proof}

Since $Z_k(T_n) \rightarrow \infty$, by the Stolz-Ces\`{a}ro theorem we know that
\[
\lim_{n\rightarrow \infty} \frac{c(T_i^k, T'_n)}{Z_k(T'_n)} = \lim_{n\rightarrow \infty} \frac{c(T_i^k, T'_{n+1}) - c(T_i^k, T'_n)}{Z_k(T'_{n+1}) - Z_k(T'_n)},
\]
if the right limit exists. By Lemma \ref{lemma: gluing inequalities} we know that
\[
\frac{c(T_i^k, T_{n+1})}{Z_k(T_{n+1}) + 2(k-2)!D(T_{n+1})^{k-3}}
\leqslant 
\frac{c(T_i^k, T'_{n+1}) - c(T_i^k, T'_n)}{Z_k(T'_{n+1}) - Z_k(T'_n)} 
\]
\[
\frac{c(T_i^k, T'_{n+1}) - c(T_i^k, T'_n)}{Z_k(T'_{n+1}) - Z_k(T'_n)}
\leqslant \frac{c(T_i^k, T_{n+1}) + 2(k-2)!D(T_{n+1})^{k-3}}{Z_k(T_{n+1})}
\]

Remark that $Z_k(T_{n+1}) \geqslant {D(T_{n+1})\choose k-1}$ (look at the number of $k$-stars). Therefore the quotient $D(T_{n+1})^{k-3}/Z_k(T_{n+1})\rightarrow 0$ as $n\rightarrow \infty$ and the conclusion follows.     
\end{proof}

\begin{proof2}

Let $T_n$ be a sequence of trees with $|T_n| \rightarrow \infty$ whose sequence of profiles converges to $p$. Denote $|T_n| = v_n$ and let $d_{n, 1}$, $d_{n, 2}$, ..., $d_{n,v_n}$ denote the degree sequence of $T_n$. 

We can assume that $D(T_n)$ is non-decreasing by restricting to a subsequence. Then, Lemma \ref{lemma: new sequence through gluing} allows us to assume that $\diam(T_n) \geqslant 2k$. It is not hard to check that in this case 
\[
P_k(T_n) \geqslant v_n - k.
\]

Let $0< \alpha < 1$ such that $\frac{\epsilon}{1 - \alpha} = k -1$. It is enough to prove that
\[
\frac{R_k(T_n)}{Z_k(T_n) + k} \leqslant e(k- 1)!(k-1)^\epsilon\left(\frac{P_k(T_n) + k}{Z_k(T_n) + k}\right)^\alpha, 
\]
which is equivalent to proving
\[
(R_k(T_n))^\frac{1}{1 - \alpha} \leqslant (e(k- 1)!)^{\frac{1}{1- \alpha}}(k-1)^{k-1} (P_k(T_n) + k)^\frac{\alpha}{1-\alpha}(Z_k(T_n) + k).
\]

Proposition \ref{proposition: upperbound non-stars} and Jensen's inequality imply
\begin{align*}
(R_k(T_n))^\frac{1}{1- \alpha} &\leqslant (e(k-1)!)^\frac{1}{1 - \alpha}\left(\sum_{i = 1}^{v_n}d_{n,i}^\epsilon\right)^\frac{1}{1 - \alpha} \leqslant (e(k-1)!)^\frac{1}{1 - \alpha} (v_n)^\frac{\alpha}{1 - \alpha}\sum_{i = 1}^{v_n}d_{n,i}^\frac{\epsilon}{1- \alpha}\\
&\leqslant (e(k-1)!)^\frac{1}{1 - \alpha} (P_k(T_n) + k)^\frac{\alpha}{1- \alpha} \sum_{i = 1}^{v_n}d_{n,i}^{k-1}\\ &\leqslant (e(k-1)!)^\frac{1}{1 - \alpha} (P_k(T_n) + k)^\frac{\alpha}{1- \alpha} (k-1)^{k-1}(P_k(T_n) + k + S_k(T_n))\\
&\leqslant (e(k-1)!)^\frac{1}{1 - \alpha} (P_k(T_n) + k)^\frac{\alpha}{1- \alpha} (k-1)^{k-1}(Z_k(T_n) + k),
\end{align*}
for the third inequality we used that $d^{k-1} \leq (k-1)^{k-1}{d \choose k-1}$ for all $d\geq k-1$. 
\end{proof2}

In the case $k = 5$ we can improve this result by improving the upper-bound on $R_5(T)$. 

\begin{proposition}
\label{prop: y bound}
Let $T$ be a tree with degree sequence $d_1$, $d_2$, ..., $d_n$. Then \[
Y(T) \leqslant 2\sum_{v = 1}^{n} (d_v)^{2 + \sqrt{3}}.
\]
In addition, $2+ \sqrt{3}$ is the optimal exponent.
\end{proposition}
\begin{proof}
Let us consider $T$'s orientation obtained by pointing all the edges away from vertex $v_1$. Denote 
\[
A(v) = \{u\in V\colon (v,u) \in E \text{ and it points away from } v\}.
\]

It is easy to check that
\[
Y(T) = \sum_{v\in V}\sum_{u\in A(v)}{d_u - 1 \choose 2}(d_v - 1) + {d_v - 1 \choose 2}(d_u - 1). 
\]
Now, fix $v \in V$ and bound from above the sum
\[\sum_{u \in A(v)} {d_u - 1 \choose 2}(d_v - 1) + {d_v - 1 \choose 2}(d_u - 1) \leqslant \frac{1}{2}\sum_{u\in A(v)}(d_u)^2d_v + (d_v)^2d_u\]

We split the summation as follows. Let $A(v)^+ = \{u\in A(v)\colon d_v > (d_u)^{\alpha}\}$, where $\alpha > 1$ to be determined later, and let $A(v)^- = A(v)\setminus A(v)^+$. Also, let 
\begin{align*}
S_+ &= \frac{1}{2}\sum_{u\in A(v)^+} (d_u)^2d_v + (d_v)^2d_u \leqslant \frac{1}{2}\sum_{u\in A(v)^+} (d_v)^{\frac{2}{\alpha} + 1} + (d_v)^{2 + \frac{1}{\alpha}}\leqslant |A(v)^+|(d_v)^{2+\frac{1}{\alpha}} \leqslant (d_v)^{3 + \frac{1}{\alpha}}\\
S_- &= \frac{1}{2}\sum_{u\in A(v)^-}(d_u)^2d_v + (d_v)^2d_u \leqslant \frac{1}{2}\sum_{u\in A(v)^-} (d_u)^{2 + \alpha} + (d_u)^{1 + 2\alpha} \leqslant \sum_{u\in A(v)} (d_u)^{1 + 2\alpha},
\end{align*}
where we used $|A(v)^+|\leqslant |A(v)| \leqslant d_v$ and $A(v)^-\subset A(v)$.

To obtain the best exponent to bound both $S_+$ and $S_-$ we impose $3 + \frac{1}{\alpha} = 1 + 2\alpha$. Hence, $\alpha = \frac{1}{2}(1 + \sqrt{3})$. We obtain:
\[
\sum_{u \in A(v)} {d_u - 1 \choose 2}(d_v - 1) + {d_v - 1 \choose 2}(d_u - 1) \leqslant S_+ + S_- \leqslant (d_v)^{2 + \sqrt{3}} + \sum_{u\in A(v)} (d_u)^{2 + \sqrt{3}}. 
\]
Summing over $v\in V$ yields the first part of the proposition. 

Now, we prove the second part of the proposition. Consider a tree $T$ of depth $2$, rooted at a node $r$. Suppose the degree of $r$ is $d$ and each of $r$'s children has degree $k$. Then, it is immediate that
\[
Y(T) = d\left({d - 1\choose 2}(k-1) + {k-1 \choose 2}(d-1)\right)
\]
Suppose $k = \lfloor d^\alpha \rfloor$ for some $\alpha >0$ and suppose there exist $C, \epsilon >0$ such that: 
\[
Y(T) \leqslant C\sum_{v\in v} (d_v)^\epsilon = Cd^\epsilon + Cdk^\epsilon + Cdk \leqslant Cd^\epsilon + Cd^{1 + \epsilon\alpha} + Cd^{1 + \alpha}.
\]

Then, by letting $d$ tend to infinity we see that all the following inequalities must hold $3 + \alpha \leqslant \max{\{\epsilon, 1 + \epsilon\alpha, 1 + \alpha\}}$, and $2\alpha + 2 \leqslant \max{\{\epsilon, 1 + \epsilon\alpha, 1 + \alpha\}}$. In other words, there must be no $\alpha$ that violates these inequalities. Suppose the first one is violated. Then, $\alpha > \epsilon - 3$ and $\alpha < \frac{2}{\epsilon - 1}$. So if we want to be no $\alpha$ with these properties we need $\epsilon - 3 = \frac{2}{\epsilon - 1}$, which yields $\epsilon = 2 + \sqrt{3}$. 
\end{proof}  

\begin{proof3}
The argument is analogous to the one used in the proof of the first part of the theorem, but uses the stronger bound on $R_5(T)$ given by Proposition \ref{prop: y bound}. 
\end{proof3}

\subsection*{Acknowledgements}
K. E. and C. S. thank Jonathan Noel and Natasha Morrison for helpful discussions.

\bibliographystyle{plain}
\bibliography{BEMS14}

\begin{thebibliography}{1}

\bibitem{BL13}
S.~Bubeck and N.~Linial.
\newblock On the local profiles of trees.
\newblock {\em Journal of Graph Theory}, 81(2):109--119, 2016.

\bibitem{CSW15}
{\'E}.~{Czabarka}, L.~A. {Sz{\'e}kely}, and S.~{Wagner}.
\newblock {Paths vs. stars in the local profile of trees}.
\newblock {\em ArXiv e-prints}, 2015.

\end{thebibliography}
\end{document}